\documentclass[12pt]{article} 

\usepackage[utf8]{inputenc} 

\usepackage{geometry} 
\usepackage{graphicx} 

\usepackage{booktabs} 
\usepackage{array} 
\usepackage{paralist} 
\usepackage{verbatim} 
\usepackage{subfig} 

\usepackage{amsthm}
\usepackage{amsmath}
\usepackage{amssymb}
\usepackage{amscd}
\usepackage[all, cmtip]{xy}

\theoremstyle{plain}
\newtheorem{thm}{Theorem}[section]
\newcommand{\thistheoremname}{}
\newtheorem{genericthm}[thm]{\thistheoremname}

\newtheorem{lem}[thm]{Lemma}

\newtheorem{cor}[thm]{Corollary}

\theoremstyle{definition}

\theoremstyle{remark}
\newtheorem{rem}{Remark}[section]

\title{The Milnor-Palamodov Theorem for Functions on Isolated Hypersurface Singularities}
\author{Konstantinos Kourliouros\\
ICMC-USP\footnote{Avenida Trabalhador San Carlense 400, Centro, S\~ao Carlos, SP, Brazil.
k.kourliouros@gmail.com}
}

\begin{document}

\maketitle

\begin{abstract}
In this note we give a simple proof of the following relative analog of the well known Milnor-Palamodov theorem: the Bruce-Roberts number of a function relative to an isolated hypersurface singularity is equal to its topological Milnor number (the rank of a certain relative (co)homology group) if and only if the hypersurface singularity is quasihomogeneous. The proof relies on an interpretation of the Bruce-Roberts number in terms of differential forms and the L\^e-Greuel formula.
\medskip


\medskip

\end{abstract}


\section{Introduction}

Let $f:(\mathbb{C}^{n+1},0)\rightarrow (\mathbb{C},0)$ be an analytic function germ with an isolated critical point at the origin. The classical Milnor-Palamodov theorem \cite{Mil}, \cite{Pal} is a statement on the equality between an analytic invariant of the singularity, namely, the dimension of the local algebra $\mathcal{O}_{n+1}/J_f$ (where $J_f$ is the ideal generated by the partial derivatives of $f$), with a topological invariant, namely, the number of $n$-dimensional spheres in the bouquet decomposition of the Milnor fiber $X_t$ of $f$, or what is equivalent, the rank of the middle (co)homology group $H_n(X_t;\mathbb{Z})$. This common number is now well known as the Milnor number $\mu(f)$ of the singularity $f$ and is one of the most important analytic-topological invariants of the singularity.

The Milnor-Palamodov theeorem has been generalised in a number of cases and in particular for the isolated complete intersection singularities (the famous L\^e-Greuel formula \cite{Gr}, \cite{Le1}), as well as for the so called isolated boundary singularities of V. I. Arnol'd \cite{A1} (see \cite{Kou} for a proof), consisting of pairs $(f,V)$ where $f$ is an isolated singularity and $V$ is a smooth hypersurface such that the restriction $f_V$ of $f$ on $V$ also has an isolated singularity. 

Here we will consider a case interpolating between the previous two, i.e. pairs $(f,V)$ where again $f$ is an isolated singularity, and $V$ is a hypersurface which also has an isolated singularity at the origin, and is such that the restriction $f_V$ of $f$ on $V$ defines an isolated complete intersection at the origin. 

The analytic invariant generalising the dimension of the local algebra of the singularity in the case of pairs $(f,V)$ is the so called Bruce-Roberts number $\mu_{BR}(f,V)$ (introduced by J. W. Bruce and R. M. Roberts in \cite{BR} for arbitrary analytic sets $V$, but already studied earlier by O. V. Lyasko \cite{Ly} for the hypersurface case), which is the dimension of the local algebra $\mathcal{O}_{n+1}/J_f(V)$, where $J_f(V)$ is the ideal generated by the derivatives of $f$ along ``logarithmic'' vector fields, i.e. tangent to the smooth part of $V$. The corresponding topological invariant is the relative Milnor number $\mu(f,V)$, defined as the rank of the middle relative (co)homology group $H_n(X_t,X_t\cap V_s;\mathbb{Z})$ (where $V_s$ is the Milnor fiber of a smoothing of $V$), which is in turn equal to the sum of the ranks of $H_n(X_t;\mathbb{Z})$ and of  $H_{n-1}(X_t\cap V_s;\mathbb{Z})$, the latter being also equal to the number of $(n-1)$-dimensional spheres in the bouquet decomposition of the isolated complete intersection $f_V$ (this is a well known theorem due to H. Hamm \cite{Ha}).  

The main result of the paper (Theorem \ref{MP}) is the following relative analog of the Milnor-Palamodov theorem: the Bruce-Roberts number $\mu_{BR}(f,V)$ of the pair $(f,V)$ is equal to its relative topological Milnor number $\mu(f,V)$  if and only if the hypersurface singularity $V$ is quasihomogeneous. In particular the following formula holds:
\begin{equation}
\label{for-0}
\mu_{BR}(f,V)=\mu(f,V)+q(V),
\end{equation}  
where $q(V)$ is the degree of non-quasihomogeneity of the hypersurface $V$ (i.e. the difference of the Milnor and Tjurina numbers of $V$). 

The ``if" part of this theorem has been obtained rather recently in \cite{Ba}, where the authors give a proof using generators of the ideal $J_f(V)$ and certain  properties of the logarithmic characteristic variety, introduced by K. Saito in \cite{Sa1} and also used extensively by Bruce and Roberts in \cite{BR}. For quasihomogeneous $V$, they obtain nice properties of this logarithmic variety (Cohen-Macaulay) and they also give a relation of the Bruce-Roberts number with the so called Euler obstruction \cite{Grj}, \cite{Se}.   

In contrast to this ``microlocal'' approach of Saito's theory, we give here a more ``global'', but still elementary proof of formula (\ref{for-0}), which relies on a simple interpretation of the Bruce-Roberts number in terms of relative differential forms, the consideration of the so called torsion differentials, and an application of the well known L\^e-Greuel formula, given in its original form by G. M. Greuel \cite{Gr}.

We remark that the proof presented here, despite the fact of being elementary, it has the drawback that it is non-canonical, in the sense that it does not reveal the natural canonical relation that exists between the local algebra $\mathcal{O}_{n+1}/J_f(V)$ of the singularity with the relative cohomology $H^n(X_t,X_t\cap V_s)$ of the pair of Milnor fibers. To obtain such a canonical relation one has to introduce certain relative Brieskorn modules associated to the pair $(f,V)$ along with a relative Gauss-Manin connection naturally defined on them (as in \cite{Kou} for the case of isolated boundary singularities). Then formula (\ref{for-0}) is just a manifestation of Malgrange's index formula \cite{Mal} for this relative Gauss-Manin connection which, in the present exposition, is hidden within the L\^e-Greuel formula. Such an approach, along with several more profound relations with the relative monodromy, the spectrum and the spectral pairs coming from the eventual mixed Hodge structure in the relative cohomology, will be presented in a subsequent paper.

\section{Bruce-Roberts and Topological Milnor numbers}

Let $f:(\mathbb{C}^{n+1},0)\rightarrow (\mathbb{C},0)$ be an analytic function germ with an isolated critical point at the origin and let $(V,0)\subset \mathbb{C}^{n+1}$  be a germ of a reduced analytic hypersurface, given as the zero locus $V=g^{-1}(0)$ of an analytic function germ $g:(\mathbb{C}^{n+1},0)\rightarrow (\mathbb{C},0)$ which also has an isolated singularity at the origin (we do not exclude the case where $V$ might be smooth). Denote by $f_V:(V,0)\rightarrow (\mathbb{C},0)$ the restriction of $f$ on $V$ and suppose that it is a submersion along the smooth points $V^*=V\setminus \{0\}$ of $V$. Then $f_V$ also has an isolated critical point at the origin (in the stratified sense) and in particular, the intersection $f^{-1}(0)\cap V=f_V^{-1}(0)$ defines an isolated complete intersection singularity of dimension $n-1$. 

Denote now by $\Theta(V)\subset \Theta$ the submodule of vector fields at the origin which are tangent to the smooth part $V^*$ of $V$ (also known as logarithmic vector fields due to K. Saito \cite{Sa1}):
\[X\in \Theta(V)\Longleftrightarrow X(g)\in <g>\Longleftrightarrow X|_{V^*}\in TV^*.\]
Denote also by
\[J_f(V):=\{L_Xf/X\in \Theta(V)\}\]
the ideal generated by the derivatives of $f$ along logarithmic vector fields. Then, the local algebra of the singularity $(f,V)$ is defined as the quotient 
\[\mathcal{Q}_f(V):=\frac{\mathcal{O}_{n+1}}{J_f(V)}.\]
The dimension of this algebra, which is obviously finite dimensional, is usually called (by several authors) the Bruce-Roberts number $\mu_{BR}(f,V)$ of the singularity $(f,V)$:
\[\mu_{BR}(f,V):=\dim_{\mathbb{C}}\mathcal{Q}_f(V).\]

Obviously the Bruce-Roberts number of the pair $(f,V)$ is always bigger (or equal) to the ordinary Milnor number of $f$. In particular, the inclusion $J_f(V)\subseteq J_f$ (where $J_f$ is the ideal of partial derivatives of $f$) induces a projection of local algebras $\mathcal{Q}_f(V)\rightarrow \mathcal{Q}_f$ (where $\mathcal{Q}_f=\mathcal{O}_{n+1}/J_f$ is the local algebra of $f$ as in the introduction), whose kernel is the quotient of ideals:
\[\mathcal{Q}_{f_V}:=\frac{J_f}{J_f(V)},\]
and whose dimension we denote by 
\[\mu_{BR}(f_V):=\dim_{\mathbb{C}}\mathcal{Q}_{f_V}.\]
Thus we obtain a short exact sequence of local algebras:
\begin{equation}
\label{ses-la-1}
0\rightarrow \mathcal{Q}_{f_V}\rightarrow \mathcal{Q}_f(V)\rightarrow \mathcal{Q}_f\rightarrow 0,
\end{equation} 
which implies the following relation:
\begin{equation}
\label{fun-rel-2}
\mu_{BR}(f,V)=\mu(f)+\mu_{BR}(f_V).
\end{equation}

Concerning the topology of pairs $(f,V)$, denote by $X_t$ the Milnor fiber of $f$ and by $X_t\cap V_s$ the Milnor fiber of the isolated complete intersection singularity $(f,g)$. By Milnor's theorem \cite{Mil} the fiber $X_t$ has the homotopy type of a bouquet of $n$-dimensional spheres, whose number is equal to the Milnor number $\mu(f)$ of $f$. By Hamm's theorem \cite{Ha}, generalising Milnor's fibration theorem for the isolated complete intersection case, the fiber $X_t\cap V_s$ also has the homotopy type of a bouquet of $(n-1)$-dimensional spheres, whose number is equal to the Milnor number of the isolated complete intersection $(f,g)$ or, what is equivalent, of the function $f_V$. We denote this number by $\mu(f_V)$. It follows from this that the natural long exact sequence in homology obtained by the embedding $X_t\cap V_s\subset X_t$, reduces to the short exact sequence (it can be considered with integer coefficients):
\begin{equation}
\label{ses-coh-1}
0\rightarrow H_{n}(X_t)\rightarrow H_n(X_t,X_t\cap V_s)\rightarrow H_{n-1}(X_t\cap V_s)\rightarrow 0,
\end{equation}    
From this we also obtain the following fundamental relation of (topological) Milnor numbers:
\begin{equation}
\label{fun-rel-1}
\mu(f,V)=\mu(f)+\mu(f_V),
\end{equation}
where we denote by $\mu(f,V)$ the rank of the middle relative homology group:
\[\mu(f,V):=\mbox{rank}_{\mathbb{Z}}H_n(X_t,X_t\cap V_s).\]

The relation of the Bruce-Roberts numbers in (\ref{fun-rel-2}) with the topological Milnor numbers in (\ref{fun-rel-1}) is explicated by the following relative analog of Milnor-Palamodov theorem: 

\begin{thm}
\label{MP}
Let $(f,V)$ be an isolated singularity. Then the following identity holds:
\begin{equation}
\label{MP-id}
\mu_{BR}(f,V)=\mu(f,V)+q(V)=\mu(f)+\mu(f_V)+\mu(V)-\tau(V),
\end{equation} 
where $\mu(V)$ is the Milnor number of $V$, $\tau(V)$ is its Tjurina number, and $q(V)=\mu(V)-\tau(V)$ is the degree of non-quasihomogeneity of $V$.
\end{thm}

\begin{rem}
\label{rem-3}
Formula (\ref{MP-id}) has already been obtained for the planar case $n=1$, in a different form, by C. T. C. Wall in \cite{W}.
\end{rem}

The fact that the number $q(V)$ can be be interpreted as the degree of non-quasihomogeneity of the singularity $V$ is a well known theorem of K. Saito \cite{Sa}. From this it follows:

\begin{cor}
\label{cor-1}
Let $(f,V)$ be an isolated singularity. Then
\[\mu_{BR}(f,V)=\mu(f,V)=\mu(f)+\mu(f_V)\]
if and only if the singularity $V$ is quasihomogeneous. In that case, the Bruce-Roberts number is a topological invariant of the singularity.
\end{cor} 

\begin{rem}
\label{rem-4}
The ``if'' part of the above corollary has been obtained in \cite{Ba}.
\end{rem}

\section{Proof of the Theorem}

The proof of the theorem relies on an interpretation of the Bruce-Roberts numbers in terms of differential forms, the consideration of the so called torsion differentials, and the well known L\^e-Greuel formula.

To start, let $\Omega^{\bullet}(V^*)\subset \Omega^{\bullet}$ be the subcomplex of germs of holomorphic forms which vanish when restricted (by pullback) to the tangent bundle of the smooth part $V^*$ of the hypersurface $V$. We call the quotient complex
\[\widetilde{\Omega}^{\bullet}_V:=\frac{\Omega^{\bullet}}{\Omega^{\bullet}(V^*)}\]
the Ferrari complex of $V$, since it was introduced by A. Ferrari in \cite{Fe} (but also used extensively by T. Bloom, M. Herrera \cite{BH}, G. M. Greuel \cite{Gr} and others).  The reason for introducing this complex in our context becomes apparent due to the following:

\begin{lem}
\label{lem-1}
Multiplication by a volume form $\omega$ gives an isomorphism:
\[J_f(V)\stackrel{\cdot \omega}{\simeq}df\wedge \Omega^n(V^*).\]
\end{lem} 
\begin{proof}
It follows immediately by the identity:
\[L_Xf\cdot \omega=df\wedge (X\lrcorner \omega)\]
and the fact that $\omega$ induces an isomorphism
\[\Theta(V)\stackrel{\cdot \omega}{\simeq} \Omega^n(V^*),\]
so that $a=X\lrcorner \omega \in \Omega^n(V^*)$ if and only if $X\in \Theta(V)$. To see that the latter isomorphism is true, one may multiply $a=X\lrcorner \omega$ with $dg\wedge$:
\[dg\wedge a=dg\wedge (X\lrcorner \omega)=L_Xg\cdot \omega,\] 
and notice that (by the Poincar\'e residue short exact sequence for $V^*$) there is an identification:
\[\Omega^{n}(V^*)=\{a\in \Omega^n/dg\wedge a \in g\Omega^{n+1}\}.\]
\end{proof}

From this we obtain the following interpretation of short exact sequence (\ref{ses-la-1}) in terms of differential forms:
\begin{lem}
\label{lem-2} 
The short exact sequence of local algebras (\ref{ses-la-1}) is isomorphic to the short exact sequence:
\begin{equation}
\label{ses-df-1}
0\rightarrow \widetilde{\Omega}_{f_V}\stackrel{df\wedge}{\rightarrow}\Omega_f(V^*)\rightarrow \Omega_f\rightarrow 0,
\end{equation}
where:
\[\Omega_f:=\frac{\Omega^{n+1}}{df\wedge \Omega^n}, \quad \Omega_f(V^*):=\frac{\Omega^{n+1}}{df\wedge \Omega^n(V^*)},\]
\[\widetilde{\Omega}_{f_V}:=\frac{\widetilde{\Omega}^n_V}{df\wedge \widetilde{\Omega}^{n-1}_V}\simeq \frac{\Omega^n}{df\wedge \Omega^{n-1}+\Omega^n(V^*)}.\]
\end{lem}
\begin{proof}
Since $df\wedge \Omega^n(V^*)\subset df\wedge \Omega^n$ we obtain a short exact sequence:
\[0\rightarrow \frac{df\wedge \Omega^n}{df\wedge \Omega^n(V^*)}\rightarrow \frac{\Omega^{n+1}}{df\wedge \Omega^n(V^*)}\rightarrow \frac{\Omega^{n+1}}{df\wedge \Omega^{n}}\rightarrow 0,\]
which, by the previous Lemma \ref{lem-1}, becomes isomorphic to (\ref{ses-la-1}) after multiplication with a volume form. To obtain (\ref{ses-df-1}) it suffices to notice that 
\[\frac{df\wedge \Omega^n}{df\wedge \Omega^{n}(V^*)}\simeq df\wedge \big{(}\frac{\Omega^{n}}{df\wedge \Omega^{n-1}+\Omega^n(V^*)} \big{)}\simeq df\wedge \widetilde{\Omega}_{f_V}.\]
\end{proof}

\begin{rem}
\label{rem-5}
From the lemma we obtain in particular the identity (\ref{fun-rel-2}):
\[\mu_{BR}(f,V)=\mu(f)+\mu_{BR}(f_V),\]
where:
\[\mu(f)=\dim_{\mathbb{C}}\Omega_f,\quad \mu_{BR}(f,V)=\dim_{\mathbb{C}}\Omega_f(V^*), \quad \]
\[\mu_{BR}(f_V)=\dim_{\mathbb{C}}\widetilde{\Omega}_{f_V}.\]
Moreover, the terms appearing above are the last terms of the obvious relative de Rham complexes but this is irrelevant in the present exposition.
\end{rem}

To continue with the proof we will need to consider also the well known complex of K\"ahler differentials. Denote by $\Omega^{\bullet}(V)\subset \Omega^{\bullet}$ the subcomplex of holomorphic forms which vanish identically along the points of $V$:
\[\Omega^{\bullet}(V):=dg\wedge \Omega^{\bullet-1}+g\Omega^{\bullet}.\]
The quotient complex:
\[\Omega^{\bullet}_V:=\frac{\Omega^{\bullet}}{\Omega^{\bullet}(V)}\]
is usually called the Grauert-Grothendieck (or K\"ahler) complex of $V$. Since $\Omega^{\bullet}(V)\subseteq \Omega^{\bullet}(V^*)$ (obvious) we have a natural projection $\Omega^{\bullet}_V\rightarrow \widetilde{\Omega}^{\bullet}_V$ and in particular a short exact sequence of complexes:
\begin{equation}
\label{ses-tor}
0\rightarrow T^{\bullet}_V\rightarrow \Omega^{\bullet}_V\rightarrow \widetilde{\Omega}^{\bullet}_V\rightarrow 0,
\end{equation}
where the kernel complex:
\[T^{\bullet}_V=\frac{\Omega^{\bullet}(V^*)}{\Omega^{\bullet}(V)}\]
can be identified with the torsion subcomplex $\mbox{Tor}\Omega^{\bullet}_V$ of $\Omega^{\bullet}_V$ (here is where we need $V$ to be reduced). Indeed, the complexes $\Omega^{\bullet}_V$ and $\widetilde{\Omega}^{\bullet}_V$ are equal on the smooth part $V^*$ of $V$ and in particular, any torsion element vanishes on the smooth part of $V$, so $\mbox{Tor}\Omega^{\bullet}_V\subseteq T^{\bullet}_V$. On the other hand, any element in the kernel complex is obviously torsion and thus $T^{\bullet}_V\subseteq \mbox{Tor}\Omega^{\bullet}_V$, from which we obtain the identifications:
\[T^{\bullet}_V\simeq \mbox{Tor}\Omega^{\bullet}_V, \quad \widetilde{\Omega}^{\bullet}_V\simeq \frac{\Omega^{\bullet}_V}{\mbox{Tor}\Omega^{\bullet}_V}.\]
The following lemma concerning the dimension of the space of these torsion differentials will be very useful in what follows: 

\begin{lem}[G. M. Greuel \cite{Gr}]
\label{thm-coh-V}
\noindent

\begin{itemize}
\item[(i)] $T^p_V=0$ for $p<n$ and  
\item[(ii)] $\dim_{\mathbb{C}}T^n_V=\tau(V)$, where $\tau(V)$ is the Tjurina number of $V$.
\end{itemize}
\end{lem}

\begin{rem}
Torsion differentials are sometimes interpreted in terms of local cohomology (c.f. \cite{Gr}, \cite{Loo})
\end{rem}

Now, following Greuel \cite{Gr} we can define, using the K\"ahler complex of $V$ one more module of relative forms for the complete intersection $f_V$:
\[\Omega_{f_V}:=\frac{\Omega^n_V}{df\wedge \Omega^{n-1}_V}\simeq \frac{\Omega^{n}}{df\wedge \Omega^{n-1}+\Omega^n(V)}=\]
\[=\frac{\Omega^{n}}{df\wedge \Omega^{n-1}+dg\wedge \Omega^{n-1}+g\Omega^n}.\]
Denote its dimension by 
\[\mu_G(f_V):=\dim_{\mathbb{C}}\Omega_{f_V}.\]
The following theorem is the famous Greuel formula in its original form (also proved by L\^e D. Tr\'ang in a different form \cite{Le1}) relating the Milnor number $\mu(f_V)$ of the isolated complete intersection $f_V$ with the analytic invariant $\mu_G(f_V)$:

\begin{lem}[G. M. Greuel \cite{Gr}]
\label{gr-le}
For the isolated complete intersection $f_V$ the following formula holds:
\begin{equation}
\label{gr-le-for}
\mu_G(f_V)=\mu(f_V)+\mu(V).
\end{equation}
\end{lem}


Finally, we will need the following lemma which relates the Bruce-Roberts number $\mu_{BR}(f_V)$ with Greuel's invariant $\mu_G(f_V)$ defined above:

\begin{lem}
\label{lem-3}
For the isolated complete intersection $f_V$ the following formula holds:
\begin{equation}
\label{br-n}
\mu_G(f_V)=\mu_{BR}(f_V)+\tau(V).
\end{equation}
\end{lem}
\begin{proof}
The natural projection $\Omega^n_V\rightarrow \widetilde{\Omega}^n_V$ gives another natural projection of relative forms $\Omega_{f_V}\rightarrow \widetilde{\Omega}_{f_V}$ whose kernel is again isomorphic to $T^n_V$; indeed, by Lemma \ref{thm-coh-V}-(i), $T^{n-1}_V=0$, which implies that $\Omega^{n-1}_V\simeq \widetilde{\Omega}^{n-1}_V$ and also $df\wedge \Omega^{n-1}_V\simeq df\wedge \widetilde{\Omega}^{n-1}_V$.  From this we obtain a short exact sequence:
\[0\rightarrow T^n_V\rightarrow \Omega_{f_V}\rightarrow \widetilde{\Omega}_{f_V}\rightarrow 0,\] 
and the result follows from Lemma \ref{thm-coh-V}-(ii), according to which $\dim_{\mathbb{C}}T^n_V=\tau(V)$.
\end{proof}

\begin{proof}[Proof of Theorem \ref{MP}]
Combining formulas (\ref{br-n}) and (\ref{gr-le-for}) we obtain:
\[\mu_{BR}(f_V)=\mu(f_V)+\mu(V)-\tau(V).\]
Substituting this formula to the fundamental relation (\ref{fun-rel-2}) we obtained the desired formula (\ref{MP-id}). 
\end{proof}

\section*{Acknowledgements}
This research has been supported by the Research Foundation of S\~ao Paulo (FAPESP), grand No: 2017/23555-19.


\begin{thebibliography}{55}   
\bibitem{A1} V. I. Arnol'd, \textit{Critical Points of Functions on a Manifold with Boundary, The simple Lie Grous $B_k$, $C_k$ and $F_4$ and Singularities of Evolutes}, Russian Mathematical Surveys, 33:5, (1978), 99-116 
\bibitem{Ba} J. J. Nu\~no Ballesteros, B. Or\'efice, J. N. Tomazella, \textit{The Bruce-Roberts Number of a Function on a Weighted Homogeneous Hypersurface}, Quart. J. Math., 64, 1,(2013), 269–280
\bibitem{BH} T. Bloom, M. Herrera,  \textit{De Rham Cohomology of an Analytic Space}, Inventiones Math. 7, 275-296, (1969) 
\bibitem{BR} J. W. Bruce, R. M. Roberts, \textit{Critical Points of Functions on Analytic Varieties}, Topology, 27, 1, (1988), 57-90
\bibitem{Fe} A. Ferrari, \textit{Cohomology and Holomorphic Differential Forms on Complex Analytic Spaces}, Ann. Scuola Norm. Sup. Pisa, Tome 24, No. 1, (1970), 65-77 
\bibitem{Grj} N. De G\'oes Grulha Jr., \textit{The Euler Obstruction and Bruce-Roberts’ Milnor Number}, Quart. J. Math. 60, (2009), 291—302
\bibitem{Gr} G. M. Greuel, \textit{Der Gauss-Manin-Zusammenhang Isolierter Singularit\"aten von Vollst\"andigen Durchschnitten}, Math. Ann., (1975), 235-266 
\bibitem{Ha} H. Hamm, \textit{Lokale Topologische Eigenschaften Komplexer R\"aume}, Math. Ann. 191, (1971), 235–252 
\bibitem{Kou} K. Kourliouros, \textit{Gauss-Manin Connections for Boundary Singularities and Isochore Deformations}, Demonstratio Matematica, 48, 2, (2015), 250-288
\bibitem{Le1} L\^e D. Tr\'ang, \textit{Calculation of Milnor Number of Isolated Singularity of Complete Intersection}, Funct. Anal. its Appl., 8, 2, (1974), 127-131 
\bibitem{Loo} E. J. N. Looijenga, \textit{Isolated Singular Points on Complete Intersections}, London Mathematical Society Lecture Notes, 77, (1984)
\bibitem{Ly} O. V. Lyasko, \textit{Classification of Critical Points of Functions on a Manifold with Singular Boundary}, Functional Analysis and Its Applications, 17:3, (1983), 187–193
\bibitem{Mal} B. Malgrange, \textit{Int\'egrales Asymptotiques et Monodromie}, Ann. Scient. Ec. Norm. Sup., 7, (1974), 405-430
\bibitem{Mil} J. Milnor, \textit{Singular Points of Complex Hypersurfaces}, Princeton University Press and the Tokyo University Press, Princeton, New Jersey, (1968)
\bibitem{Pal} V. P. Palamodov, \textit{On the Multiplicity of Holomorphic Mappings}, Funct. Anal. Appl. 1-3, (1967), 54-65
\bibitem{Sa} K. Saito, \textit{Quasihomogene Isolierte Singularit\"aten von Hyperfl\"achen}, Invent. Math. 14, (1971), 123-142
\bibitem{Sa1}K. Saito, \textit{Theory of Logarithmic Differential Forms and Logarithmic Vector Fields}, J. Fac. Sci. Univ. Tokyo Sect. 1A, Math 27, 2, (1980), 265-291
\bibitem{Se} J. Seade, M. Tibar, A. Verjovsky, \textit{Milnor Numbers and Euler Obstruction}, Bull. Braz. Math. Soc. (N.S.), 36, (2005), 275-283.
\bibitem{W} C. T. C. Wall, \textit{A Note on Relative Invariants and Determinacy of Plane Curves}, Journal of Singularities, 4, (2012), 188-195
\end{thebibliography}
\end{document}